\documentclass[10pt]{amsart}

\usepackage{amsmath}

\usepackage{amssymb}

\usepackage{graphicx}

\newtheorem{thm}{Theorem}
\newtheorem{prop}[thm]{Proposition}
\newtheorem{lemma}[thm]{Lemma}
\newtheorem{cor}[thm]{Corollary}

\theoremstyle{definition}
\newtheorem{defn}{Definition}

\theoremstyle{remark}
\newtheorem{remark}{Remark}

\def\C{\mathbb{C}}
\def\R{\mathbb{R}}

\def\Z{\mathbb{Z}}

\def\a{\mathfrak{a}}

\def\sl{\mathfrak{sl}}

\def\SL{\mathrm{SL}}

\def\he{\mathrm{Herm}}

\def\det{\mathrm{det\,}}
\def\dim{\mathrm{dim\,}}

\def\span{\mathrm{span\,}}

\begin{document}

\title[Reduction for variational
 problems on null curves]{Reduction for constrained
variational \\ problems on 3D null curves}

\author{Emilio Musso}
\address{(E. Musso) Dipartimento di Matematica Pura ed Applicata,
Universit\`a degli Studi dell'Aquila, Via Vetoio, I-67010 Coppito
(L'Aquila), Italy} \email{musso@univaq.it}

\author{Lorenzo Nicolodi}
\address{(L. Nicolodi) Di\-par\-ti\-men\-to di Ma\-te\-ma\-ti\-ca,
Uni\-ver\-si\-t\`a degli Studi di Parma, Viale G. P. Usberti 53/A,
I-43100 Parma, Italy} \email{lorenzo.nicolodi@unipr.it}

\thanks{Authors partially supported by MIUR projects:
\textit{Metriche riemanniane e variet\`a differenziali} (E.M.);
\textit{Propriet\`a geometriche delle variet\`a reali e complesse}
(L.N.); and by the GNSAGA of INDAM}

\subjclass[2000]{49F05; 58E10; 58A17}



\keywords{Null curves, invariant variational problems, extremal
trajectories, optimal control systems, moving frames, Lax
formulation, Marsden--Weinstein reduction.}

\begin{abstract}
We consider the optimal control problem for null curves in de Sitter
3-space defined by a functional which is linear in the curvature of
the trajectory. We show how techniques ba\-sed on the method of
moving frames and exterior differential systems, coupled with the
reduction procedure for systems with a Lie group of symmetries lead
to the integration by quadratures of the extremals. Explicit
solutions are found in terms of elliptic functions and integrals.
\end{abstract}

\maketitle

\section{Introduction}\label{s:intro}

Let $M^3$ be a 3-dimensional Lorentz space form and $\gamma \subset
M^3$ a null curve parametrized by the natural (pseudo-arc) parameter
$s$ which normalizes the derivative of its tangent vector field. It
is known that in general $\gamma$ admits a curvature $k_\gamma(s)$
that is a Lorentz invariant and that uniquely determines $\gamma$ up
to Lorentz transformations. We consider the variational problem on
null curves defined by the Lorentz invariant functional
\begin{equation}\label{functional}
 \mathcal{L}(\gamma) = \int_\gamma{(m + k_\gamma)}ds, \quad m\in \R,
  \end{equation}
and ask the question of determining the explicit form for the
extremal trajectories. Motivations are provided by optimal control
theory and recent work on relativistic particle models associated
with action functionals of the type above (cf. \cite{P-NucPhysB},
\cite{NR-PhysLettB}, \cite{NMMK-NucPhysB}, \cite{FGL-PhysLettB}, and
references therein).

\vskip0.2cm

From the Euler--Lagrange equation of the action it follows that the
curvature of an extremal trajectory is either constant, or an
elliptic function (possibly degenerate) of the natural parameter. In
the first case, the extremals are orbits of 1-parameter subgroups of
the group of Lorentz transformations and can be described in terms
of elementary functions \cite{FGL-IJMPA}. In the second case, we are
led to a linear system of ODEs whose coefficients are
doubly-periodic functions. By the Fuchsian theory of ODEs, and in
particular the results of Picard \cite{Pi}, the trajectories are
then expressible in terms of the Weierstrass elliptic functions
$\wp$, $\sigma$ and $\zeta$. Alternatively, we follow a general
scheme for the reduction of constrained variational problems on
homogeneous spaces. We will use techniques from optimal control
theory based on the method of moving frames and Cartan's exterior
differential systems \cite{CartanLSLII}, \cite{Gr}
\cite{Gardner1983}, \cite{Gardner1989}, coupled with the reduction
procedure for systems admitting a Lie group of symmetries extended
to this setting \cite{BG}. For other applications of this general
scheme of integration we refer to \cite{GM-JGP}, \cite{MN-Forum},
\cite{MN-adS}.

\vskip0.2cm

In this article, we determine the explicit form of the extremal
curves when the target manifold is de Sitter 3-space. In this case,
the functional \eqref{functional} is invariant under the group
$\SL(2,\C)$, which doubly covers the identity component of the
isometry group of de Sitter 3-space. The starting point of our study
is the replacement of the original variational problem on null
curves in de Sitter 3-space by an $\SL(2,\C)$-invariant variational
problem for integral curves of a control system on
$M\cong\SL(2,\C)\times \R$ defined by a suitable Pfaffian
differential ideal $(\mathcal{I},\omega)$ with an independence
condition. This is accomplished by proving the existence of a
preferred $\SL(2,\C)$-invariant frame along null curves without flex
points (cf. Section \ref{s:pre}). We then follow a general
construction due to Griffiths \cite{Gr} and carry out a calculation
to associate to the variational problem a Pfaffian differential
system $\mathcal{J}$, the \textit{Euler--Lagrange system}, whose
integral curves are stationary for the associated functional. The
Euler--Lagrange system is defined on the \textit{momentum space} $Y
\cong \SL(2,\C)\times \R^3$, which turns out to carry a contact
structure, whose characteristic curves coincide with the integral
curves of $\mathcal{J}$. As a matter of fact, in the case at hand
all extremal trajectories arise as projections of integral curves of
the Euler--Lagrange system. The theoretical reason for this is that
all the derived systems of $(\mathcal{I},\omega)$ have constant rank
(cf. \cite{Br}). Further, we show that the characteristic flow
factors over a flow in an affine 3-dimensional subspace of
$\sl(2,\C)$ and find a Lax formulation of its defining differential
equation. This implies that the momentum map induced by the
Hamiltonian action of $\SL(2,\C)$ on $Y$ is constant on solution
curves of the Euler--Lagrange system, which leads to the integration
by quadratures of the extremals (cf. Section \ref{s:integration}).

\vskip0.2cm

The paper is organized as follows. Section \ref{s:pre} gives the
details of the construction of the canonical frame along null curves
with no flex points by the method of moving frames, and defines the
Pfaffian differential system of such frames. Section \ref{s:var-pbm}
studies the action functional \eqref{functional}, introduces the
corresponding Euler--Lagrange system, and proves the constancy of
the momentum map on its integral curves. Section \ref{s:integration}
focuses on the integration procedure. It first outlines some facts
from the theory of elliptic functions, and then carries out the
explicit integration of the extremals in terms of elliptic functions
and elliptic integrals of the third kind.


\section{Preliminaries}\label{s:pre}

\subsection{The geometry of de Sitter 3-space}

Let $\he(2)$ be the four-dimensional space of $2\times 2$ Hermitian
complex matrices endowed with the Lorentz metric given by the
quadratic form $\langle X , X \rangle = -\det X$, for all $X\in
\he(2)$. De Sitter 3-space, $\mathbb{S}^3_1$, can be viewed as the
set of $2\times 2$ Hermitian matrices of determinant $-1$:
\begin{equation}
 \mathbb{S}^3_1 =\left\{X\in \mathrm{Herm}(2) \,|\,  \det X =-1\right\}
    \end{equation}
with the induced metric $g$. The special linear group $\SL(2,\C)$
acts transitively by isometries on $\mathbb{S}^3_1$ via the action
\[
  A\cdot X = AXA^{\ast},
  \]
where $A^\ast$ stands for the conjugate transpose of $A$.
The
stability subgroup at
\[
 J=\left(\begin{array}{cc}
     0&-i\\
      i&0\\
       \end{array}\right)
       \]
is the group $\SL(2,\R)$ and $\mathbb{S}^3_1$ may be described as a
Lorentzian symmetric space
\[
 \mathbb{S}^3_1 \cong \SL(2,\C)/\SL(2,\R).
  \]
The projection
\[
 \pi : \SL(2,\C) \ni A \mapsto AJ{{A}}^{\ast} \in \mathbb{S}^3_1
   \]
makes $\SL(2,\C)$ into a principal bundle with structure group
$\SL(2,\R)$.

\vskip0.2cm

Let $\Omega = \alpha +i \beta$ be the Maurer--Cartan form of
$\SL(2,\C)$, where
\begin{equation}\label{MC}
 \alpha =
  \left(
   \begin{array}{cc}
     \alpha^1_1&\alpha^1_2\\
      \alpha^2_1&-\alpha^1_1\\
       \end{array}\right),
\quad \beta = \left(
   \begin{array}{cc}
     \beta^1_1&\beta^1_2\\
      \beta^2_1&-\beta^1_1\\
       \end{array}\right).
       \end{equation}
Note that the matrix of 1-forms $\beta$ is semibasic\footnote{We
recall that a differential form $\varphi$ on the total space of a
fiber bundle $\pi : P \to B$ is said to be \textit{semibasic} if its
contraction with any vector field tangent to the fibers of $\pi$
vanishes, or equivalently, if its value at each point $p\in P$ is
the pullback via $\pi^\ast_p$ of some form at $\pi(p)\in B$. Some
authors call such a form \textit{horizontal}. A stronger condition
is that $\varphi$ be \textit{basic}, meaning that it is locally the
pullback via $\pi^\ast$ of a form on the base $B$.} for the
projection $\pi$, and that the Lorentz metric $g$ on
$\mathbb{S}^3_1$ is given by
\[
  g = \left(\beta^1_1\right)^2 - \beta^2_1\beta^1_2.
   \]
The matrix $\alpha$ amounts to the Levi-Civita (spinor) connection
of $g$. The Maurer-Cartan equations of $\SL(2,\C)$, or the structure
equations, are given by:
\[
\begin{cases}
 d\alpha^1_1 = -\alpha^1_2 \wedge\alpha^2_1 + \beta^1_2 \wedge\beta^2_1\\
 d\alpha^2_1 = 2\alpha^1_1 \wedge\alpha^2_1 -2 \beta^1_1 \wedge\beta^2_1\\
 d\alpha^1_2 = -2\alpha^1_1 \wedge\alpha^1_2 + 2\beta^1_1 \wedge\beta^1_2
 \end{cases}
  \]
\[
 \begin{cases}
 d\beta^1_1 = -\beta^1_2 \wedge\alpha^2_1 + \beta^2_1 \wedge\alpha^1_2\\
 d\beta^2_1 = 2\beta^1_1 \wedge\alpha^2_1 -2 \beta^2_1 \wedge\alpha^1_1\\
 d\beta^1_2 = -2\beta^1_1 \wedge\alpha^1_2 +2 \beta^1_2 \wedge\alpha^1_1.
 \end{cases}
   \]

\subsection{The canonical frame along a null curve}\label{ss:frame}

A smooth parametrized curve
\[
  \gamma : I  \rightarrow \mathbb{S}^3_1,
  \]
where $I$ denotes any open interval of real numbers, is
\textit{null} (or \textit{lightlike}) if the velocity vector field
$\gamma'$ is null along $\gamma$, i.e.,
$g(\gamma'(t),\gamma'(t))=0$, for each $t\in I$. We will assume
throughout that $\gamma$ has no flex points, i.e., $\gamma'(t)$ and
$\gamma''(t)$ are linearly independent, for each $t\in I$, where
$\gamma''$ denotes the covariant derivative of $\gamma'$ along the
curve.

A frame field along $\gamma$ is a smooth map $\Gamma : I \to
\SL(2,\C)$ such that $\gamma = \pi \circ \Gamma$. For any such
frame, let $\Theta = \Gamma^{\ast}\Omega$ denote the pull-back of
the Maurer--Cartan form of $\SL(2,\C)$ and write $\Theta = \phi +i
\theta$. Given a frame field along $\gamma$, any other is given by
\[
 \tilde{\Gamma} = \Gamma X
  \]
where $X : I \to \SL(2,\R)$ is a smooth map. If $\tilde{\Theta} =
\tilde{\Gamma}^{\ast}\Omega = \tilde{\phi} +i \tilde{\theta}$, then
\begin{equation}\label{transf-rule}
 \tilde{\Theta} = X^{-1}\Theta X + X^{-1} dX.
  \end{equation}

A frame field $\Gamma : I \to \SL(2,\C)$ along $\gamma$ is said of
\textit{first order} if
\begin{equation}
 \theta^1_1 = \theta^1_2 =0, \quad \theta^2_1 \neq 0.
   \end{equation}
It easily seen that first order frame fields exist locally. If
$\Gamma : I \to \SL(2,\C)$ is a first-order frame along $\gamma$,
then any other is given by $\tilde{\Gamma} = \Gamma X$, where $X : I
\to G_1\subset \SL(2,\R)$ is a smooth map, and
\[
  G_1 = \left\{ \left(\begin{array}{cc}
     a&0\\
      c&a^{-1}\\
       \end{array}\right) \, : \, a\neq 0, c\in\R\right\}.
        \]
According to \eqref{transf-rule}, one computes
\begin{equation}
 \tilde{\phi}^1_2 = {a^2}\phi^1_2, \quad \tilde{\theta}^2_1 = \frac{1}{a^2}\theta^2_1.
  \end{equation}
Moreover, for first-order frames the form $\phi^1_2$ is semibasic.
If the curve $\gamma$ has no flex points, then $\phi^1_2 \neq 0$. We
say that the curve has \textit{positive} or \textit{negative spin}
according as $\phi^1_2$ is a positive or negative multiple of
$\theta^2_1$.

Under our assumption, it follows from the transformation formula
\eqref{transf-rule} that there always exist local first order frames
along $\gamma$ such that
\begin{equation}\label{2nd}
 {\phi}^1_2 = \varepsilon\theta^2_1,
  \end{equation}
where $\varepsilon =\pm 1$, according as $\gamma$ has positive or
negative spin. A first order frame field is said of \textit{second
order} if it satisfies \eqref{2nd} on $I$.

A second order frame field along $\gamma$ is said a \textit{canonical}
frame if
\begin{equation}
 {\phi}^1_1 = 0.
  \end{equation}
Note that canonical frame fields exist on $I$, and that if $\Gamma$ is
a canonical frame, then any other is given by $\pm \Gamma$.

\vskip0.2cm

Summarizing, we have proved the following.

\begin{prop}\label{existenceSF}
Let $\gamma : I \subset \R \to \mathbb{S}^3_1$ be a null curve with
no flex points. Then, there exists a frame along $\gamma$,
the \textbf{canonical frame},
\[
 \Gamma : I \to \SL(2,\C)
   \]
such that
\begin{equation}\label{canonical-frame}
 \Gamma^{-1}d\Gamma =
  \left(
   \begin{array}{cc}
     0&\varepsilon\\
      k+i&0\\
       \end{array}\right)\omega,
       \end{equation}
where $\varepsilon = \pm 1$,
$\omega$ is a nowhere vanishing 1-form, the \textbf{canonical pseudo-arc
element}, and $k : I \to \R$ is a smooth function, the
\textbf{curvature} of $\gamma$. Moreover, if $\Gamma$ is a canonical
frame field along $\gamma$, then any other canonical frame field is
given by $\pm\Gamma$.
\end{prop}

\begin{remark}
Henceforth, we abuse the terminology and refer to the $\Z_2$-class
$[\Gamma] =\{\pm \Gamma\}$ as the canonical frame $\Gamma$ of a null
curve $\gamma$.
\end{remark}

\begin{remark}
Conversely, for a smooth function $k : I \to \R$, let $H (k) : I \to
\sl(2,\C)$ be
\begin{equation}\label{hamiltoniano}
 H(k) = \left(
  \begin{array}{cc}
    0&\varepsilon\\
     k+i&0\\
      \end{array}\right).
       \end{equation}
Then, by solving a linear system of ODEs, there exists a unique (up
to left multiplication)
\[
 \Gamma : I \to \SL(2,\C)
  \]
such that
\begin{equation}\label{linear-system}
 \Gamma^{-1}{\Gamma}' = H(k).
  \end{equation}
In particular, $\gamma = \Gamma J {\Gamma}^{\ast} : I \to
\mathbb{S}^3_1$ is a null curve without flex points
and with curvature $k$.

\end{remark}


\begin{remark}[Null helices]
The simplest examples are null helices, that is, null curves with
constant curvature. Such curves are orbits of 1-parameter subgroups
of $\SL(2,\C)$ (cf. Remark \ref{moment-map}) and
have been described by elementary functions in \cite{FGL-IJMPA}.

\end{remark}

\subsection{The Pfaffian system of canonical frames}\label{ss:spin-sys}

Let $(\mathcal{I}, \omega)$ be the Pfaffian differential system
on $M := \SL(2,\C) \times \R$ defined by the differential
ideal $\mathcal{I}$ generated  by the linearly independent 1-forms
\[
 \begin{cases}
  \eta^1 = \beta^1_1, \quad \eta^2 = \beta^1_2, \quad
   \eta^3 = \alpha^1_1 -\varepsilon\omega,\\
    \eta^4 = \alpha^1_1,\quad \eta^5 = \alpha^2_1-k\omega,
     \end{cases}
      \]
where
\[
 \omega := \beta^2_1
  \]
gives the independence condition $\omega \neq 0$.

 Now, let $\gamma : I \to \mathbb{S}^3_1$ be a null curve without flex points.
Then, by Proposition \ref{existenceSF}, the curve $g
=(\Gamma_\gamma,k_\gamma) : I \to M$, whose components are,
respectively, the canonical frame field along $\gamma$ and the
curvature of $\gamma$, is an integral curve of the Pfaffian system
$(\mathcal{I}, \omega)$. Conversely, if $g= (\Gamma,k) : I \to M$ is
an integral curve of the Pfaffian system $(\mathcal{I}, \omega)$,
then $\gamma = \Gamma J \Gamma^\ast : I \to \mathbb{S}^3_1$ defines
a null curve with no flex points, $\Gamma$ is the canonical frame
field along $\gamma$, and $k$ is the curvature of $\gamma$. For this
reason, null curves without flex points in $\mathbb{S}^3_1$ can be
identified  with the integral curves of the Pfaffian system
$(\mathcal{I}, \omega)$.

\begin{defn}
The Pfaffian differential system $(\mathcal{I}, \omega)$ will be
referred to as the \textbf{canonical system}.
\end{defn}

\begin{remark}
A smooth curve $g=(\Gamma, k) : I \to M$ is an integral curve of the
canonical system if and only if $\Gamma : I \to \SL(2,\C)$ is a
solution of the linear system
\[
 \Gamma^{-1}(t){\Gamma}'(t) = H(k(t)).
  \]
The function $k$ plays the role of a control. Note that if we assign
a smooth map $k: I \to \R$ and a point $A_0 \in \SL(2,\C)$, then
there exists a unique integral curve $g=(\Gamma, k)$ of the control
system satisfying the initial condition $\Gamma(t_0) = A_0$, for
$t_0 \in I$.
\end{remark}

Exterior differentiation and use of the Maurer-Cartan equations
give, modulo the algebraic ideal generated by
$\eta^1,\dots,\eta^{5}$, the \textit{quadratic equations} of
$(\mathcal{I}, \omega)$:
\begin{equation} \label{QE1}
\begin{cases}
 d\omega \equiv 2(k\eta^1 + \eta^4) \wedge \omega,\\
  d\eta^1 \equiv -(k\eta^2 + \eta^3) \wedge \omega, \\
   d\eta^2 \equiv -2\varepsilon \eta^1 \wedge \omega, \\
    d\eta^3 \equiv -2\varepsilon(k\eta^1 + 2\eta^4) \wedge \omega, \\
    d\eta^4 \equiv (\eta^2 -k\eta^3 +\varepsilon \eta^5) \wedge \omega, \\
     d\eta^5 \equiv -\left(dk +2(1+k^2)\eta^1\right) \wedge \omega.
      \end{cases}
        \end{equation}

\section{The variational problem and the Euler-Lagrange system}\label{s:var-pbm}

\subsection{The constrained variational problem}

Let $\mathcal{N}$ be the space of null curves in $\mathbb{S}^3_1$
without flex points. We consider the action functional
\begin{equation}\label{action}
 \mathcal{L}_m : \gamma \in \mathcal{N} \mapsto
  \int_{I_\gamma}{(m + k_\gamma)\omega_\gamma},
   \quad m\in \R,
    \end{equation}
where ${I_\gamma}$ is the domain of definition of the curve,
$k_\gamma$ is its curvature, and $\omega_\gamma$ the canonical
pseudo-arc element (cf. Section \ref{s:pre}). We refer to
\cite{NR-PhysLettB}, \cite{P-NucPhysB}, \cite{NMMK-NucPhysB},
\cite{FGL-PhysLettB} and the references therein for a discussion on
the particle model associated with this action functional.

\begin{defn}
A curve $\gamma\in \mathcal{N}$ is said to be an \textit{extremal
trajectory} (or simply a \textit{trajectory}) in $\mathbb{S}^3_1$ if
it is a critical point of the action functional $\mathcal{L}_m$ when
one considers compactly supported variations. The constant $m$ is
called the Lagrange multiplier of the trajectory.
\end{defn}

\begin{remark}
As usual, by a compactly supported variation of $\gamma \in
\mathcal{N}$ we mean a mapping $V : I \times (-\epsilon, \epsilon)
\to \mathbb{S}^3_1$ such that: 1) $\forall u \in (-\epsilon,
\epsilon)$, the map $\gamma_u : = V(t,u) : I \to \mathbb{S}^3_1$ is
a null curve without flex points; 2) $\gamma_0 = \gamma(t)$,
$\forall t \in I$; 3) there exists a closed interval $[a,b] \subset
I$ such that
\begin{equation}\label{var}
V(t,u) = \gamma(t), \quad \forall t \in I\setminus [a,b], \,
  \forall u \in (-\epsilon, \epsilon).
    \end{equation}
Accordingly, a curve $\gamma \in \mathcal{N}$ is an extremal
trajectory if, for every compactly supported variation $V$, we have
that
\[
 \left. \frac{d}{du}
   \left( \int_{a_V}^{b_V}\left(m+k_{\gamma_u}\right)ds_u \right)\right\vert_{u=0}
   = 0,
     \]
where $[a_V, b_V ]$ is the support of the variation, i.e., the
smallest closed interval for which \eqref{var} holds, and $ds_u$ is
the canonical pseudo-arc element of the curve $\gamma_u$.

In \cite{FGL-PhysLettB}, the authors derive the Euler--Lagrange
equation associated with \eqref{action} for null curves with
prescribed endpoints and the same canonical frame at each end.
\end{remark}

By the preceding discussion (cf. Proposition \ref{existenceSF} and
Section \ref{ss:spin-sys}), a curve $\gamma \in\mathcal{N}$ is an
extremal trajectory if and only if the pair $g = (\Gamma_\gamma,
k_\gamma)$ of its canonical frame field and curvature function
 is a critical point of the variational problem on the
space $\mathcal{V}(\mathcal{I},\omega)$ of all integral curves of
$(\mathcal{I},\omega)$ defined by the functional
\begin{equation}\label{var-action}
 \widehat{\mathcal{L}}: g \in \mathcal{V}(\mathcal{I},\omega)
  \mapsto \int_{I_{g}} g^\ast ((m+k)\omega),
   \end{equation}
when one considers compactly supported variations through integral
curves of $(\mathcal{I},\omega)$.


\begin{remark}
The replacement of the original functional by the functional
\eqref{var-action} is the starting point in the application of the
Griffiths formalism. This approach to constrained variational
problems with one independent variable provides conditions for
criticality in terms of Pfaffian differential systems and is
particularly well suited when one considers compactly supported
variations among constrained curves. More importantly, it furnishes
the appropriate setting for the explicit integration of the
extremals (cf. \cite{Gr}, \cite{Br}, \cite{BG}, \cite{Hsu} and
below).

\end{remark}

\subsection{The Euler--Lagrange system}

Associated to the functional $\widehat{\mathcal{L}}$ we will
introduce, following Griffiths \cite{Gr}, the Euler--Lagrange system
$(\mathcal{J},\omega)$ on a new manifold $Y$, which will be made
explicit below.

For this, let $Z\subset T^\ast M$ be the affine subbundle defined by
\[
  Z = (m+k)\omega + I \subset T^\ast M,
   \]
where $I$ is the subbundle of $T^\ast M$ associated to the
differential ideal $\mathcal{I}$.
The 1-forms $(\eta^1$, $\dots$, $\eta^5$, $\omega)$ induce a global
affine trivialization of $Z$, which may be identified with $M\times
\R^5$ by setting
\[
 M\times \R^5 \ni ((\Gamma,k); x_1,\dots,x_5)  \mapsto
 \omega_{|(\Gamma,k)} +{x_j\eta^j}_{|(\Gamma,k)} \in Z
 \]
(throughout we use summation convention). Thus, the Liouville (canonical)
1-form of $T^\ast M$ restricted to $Z$ is given by
\[
 \mu = (m + k)\omega + x_j\eta^j.
  \]
Exterior differentiation and use of the quadratic equations
\eqref{QE1} give
\[
\begin{split}
d\mu &\equiv dk \wedge \omega + 2(m +k)(k\eta^1 +\eta^4)
\wedge\omega
    + dx_j\wedge\eta^j\\
 &\quad -x_1(k\eta^2 +\eta^3)\wedge\omega - 2\varepsilon x_2\eta^1 \wedge\omega\\
 &\quad -2\varepsilon x_3(k\eta^1 + 2\eta^4)\wedge\omega + x_4 (\eta^2 -k\eta^3
 +\varepsilon \eta^5) \wedge\omega \\
 &\quad -x_5(dk +2(1+k^2)\eta^1)\wedge \omega
 \quad \mod \{\eta^i\wedge\eta^j\}.
 \end{split}
  \]

Next, we compute the Cartan system $\mathcal{C}(d\mu) \subset
T^{\ast}Z$ determined by the 2-form $d\mu$, i.e., the Pfaffian
system generated by the 1-forms
\[
 \left\{ i_\xi d\mu \, |\, \xi \in \mathfrak{X}(Z) \right\} \subset \Omega^1(Z).
  \]
Contracting $d\mu$ with the vector fields of the tangent frame
\[
 \left(\frac{\partial}{\partial{\omega}},\frac{\partial}{\partial{k}},
  \frac{\partial}{\partial{\eta^1}},\dots,
   \frac{\partial}{\partial{\eta^5}},\frac{\partial}{\partial{x_1}},\dots,
    \frac{\partial}{\partial{x_5}}\right)
      \]
on $Z$, dual to the coframe
\[
 \left(\omega, dk, \eta^1,\dots,\eta^5,dx_1,\dots,dx_5\right),
  \]
we find the 1-forms
\begin{eqnarray}
&& \eta^1,\dots,\eta^5, \\
 \pi_1 &=& (x_5 -1)dk, \label{p1}\\
 \pi_2 &=& (1-x_5)\omega, \label{p2} \\
  \beta_1 &=&dx_1-2\left\{km + k^2 -\varepsilon x_2 -\varepsilon kx_3
     -x_5(1+k^2)\right\}\omega, \label{beta1}\\
   \beta_2 &=&dx_2 + (kx_1 -x_4)\omega, \label{beta2} \\
    \beta_3 &=& dx_3 + (x_1 +kx_4)\omega, \label{beta3}\\
   \beta_4 &=&dx_4-\left\{2(m+k) -4\varepsilon x_3\right\}\omega, \label{beta4}\\
    \beta_5 &=& dx_5 -\varepsilon x_4\omega. \label{beta5}
     \end{eqnarray}
We have proven the following.

\begin{lemma}
The Cartan system $(\mathcal{C}(d\mu),\omega)$ associated to
$(\mathcal{I}, \omega)$ is the differential ideal on $Z\cong M\times
\R^5$ generated by
\[
 \left\{\eta^1,\dots,\eta^5, \pi_1, \pi_2,\beta_1,\dots,\beta_5\right\}
  \]
and with independence condition $\omega$.
\end{lemma}

\begin{defn}

The involutive prolongation of $(\mathcal{C}(d\mu),\omega)$ on $Z$
gives rise to a Pfaffian differential system $(\mathcal{J},\omega)$
on a submanifold $Y \subset Z$, which is called the
\textit{Euler--Lagrange differential system} associated to the
variational problem. The submanifold $Y$ is called the
\textit{momentum space}. We refer the reader to the book of
Griffiths \cite{Gr} for a discussion of how this system is derived
and for more details on Pfaffian systems.
\end{defn}

\begin{lemma}
The momentum space $Y$ is the 9-dimensional submanifold of $Z$
defined by the equations
\[
 x_5 = 1, \quad x_4 = 0, \quad x_3 = \frac{\varepsilon}{2}(m+k).
  \]
The Euler--Lagrange system $(\mathcal{J},\omega)$ is the Pfaffian
differential system on $Y$ with independence condition $\omega$
generated by the 1-forms
\[
\begin{cases}
 {\eta^1}_{|Y},\dots,{\eta^5}_{|Y},\\
  \sigma_1=  dx_1 + \left(k^2 -mk +2\varepsilon x_2 +2 \right)\omega,\\
  \sigma_2 =  dx_2 + kx_1\omega,\\
   \sigma_3=  dk + 2\varepsilon x_1 \omega,\\
    \end{cases}
       \]
Moreover,
\[
\begin{split}
 \mu_{|Y} & = \frac{1}{2}\left(m-k\right)\beta^2_1 -\frac{\varepsilon}{2}k'\beta^1_1
 +\frac{1}{2}
  \left(\frac{k''}{2} -\varepsilon k(k-m) -2\varepsilon \right)\beta^1_2\\
  &\quad + \frac{\varepsilon}{2}(m+k)\alpha^1_2 + \alpha^2_1.
    \end{split}
     \]
\end{lemma}

\begin{proof}
Let $V_1(d\mu)\hookrightarrow \mathbb{P}[T(Z)] \to Z$ be the
totality of 1-dimensional integral elements of $\mathcal{C}(d\mu)$.
In view of \eqref{p1} and \eqref{p2}, we find that
\[
 V_1(d\mu)_{|((\Gamma,k);x)} \neq \emptyset \iff x_5 = 1.
  \]
Thus, the first involutive prolongation of
$(\mathcal{C}(d\mu),\omega)$, i.e., the image $Z_1\subset Z$ of
$V_1(d\mu)$ with respect to the natural projection $V_1(d\mu)\to Z$,
is given by
\[
  Z_1 = \{((\Gamma,k);x) \in Z : x_5 = 1 \}.
  \]
Next, the restriction of $\beta_5$ to $Z_1$ takes the form $-
\varepsilon x_4 \omega$. Thus the second involutive prolongation
$Z_2$ is characterized by the equations
\[
  x_5 = 1, \quad x_4 = 0.
   \]
Considering then the restriction of $\beta_4$ to $Z_2$ yields the
equations
\[
  x_5 = 1, \quad x_4 = 0, \quad x_3 = \frac{\varepsilon}{2}(m+k),
   \]
which define the third involutive prolongation $Z_3$. Now, the
restriction $\mathcal{C}_3(d\mu)$ to $Z_3$ of $\mathcal{C}(d\mu)$ is
generated by the 1-forms $\eta^1,\dots,\eta^5$ and
\begin{eqnarray*}
 \sigma_1 &=& dx_1 + \left(k^2 -mk +2\varepsilon x_2 +2 \right)\omega, \\
  \sigma_2 &=& dx_2 + kx_1\omega,\\
   \sigma_3 &=& dk + 2\varepsilon x_1 \omega.
    \end{eqnarray*}
This implies that there exists an integral element of $V_1(d\mu)$
over each point of $Z_3$, i.e., $V_1(d\mu)_{p} \neq \emptyset $, for
each $p\in Z_3$. Hence $Y := Z_3$ and $(\mathcal{J},\omega) :=
(\mathcal{C}_3(d\mu),\omega)$ is the involutive prolongation of the
Cartan system $(\mathcal{C}(d\mu),\omega)$.

\end{proof}

\begin{remark}
The importance of this construction is that the natural projection
$\pi_Y : Y \to M$ maps integral curves of the Euler-Lagrange system
to extremals of the variational problem associated to
$(M,\mathcal{I})$. The converse is not true in general. However, it
is known to be true if all the derived systems of
$(\mathcal{I},\omega)$ are of constant rank (cf. \cite{Br},
\cite{Hsu}). In our case, one can easily check, using \eqref{QE1},
that all the derived systems of $(\mathcal{I},\omega)$ have indeed
constant rank, so that all the extremals do arise as projections of
integral curves of the Euler--Lagrange system (see also Section
\ref{ss:e-l-eq}).
\end{remark}

\begin{remark}
A direct calculation shows that
\begin{equation}\label{contact-con}
 \mu_{|Y} \wedge (d\mu_{|Y})^4 \neq 0
  \end{equation}
on $Y$, i.e., the variational problem is nondegenerate.\footnote{A
variational problem is said to be \textit{nondegenerate} in case
\[
 \dim Y =2m+1 \quad \text{and}\quad
   \mu_{|Y} \wedge (d\mu_{|Y})^m \neq 0.
    \]
Let ${V}(\mathcal{J},\omega)$ and
${V}(\mathcal{C}(d\mu{|Y}),\omega)$ denote the set of integral
elements of the Euler--lagrange system and of the Cartan system. For
nondegenerate problems we have
${V}(\mathcal{J},\omega)={V}(\mathcal{C}(d\mu{|Y}),\omega)$, whereas
in general we only have inclusion ${V}(\mathcal{J},\omega)\subset
{V}(\mathcal{C}(d\mu{|Y},\omega)$ (cf. \cite{Gr}, p. 84). For a discussion
on the relation between the classical Legendre transform and the
construction of the Euler--Lagrange system on the momentum space,
with special attention to the nondegeneracy condition, we refer the reader
to \cite{Gr}, Chapter I, Section e). } This
implies that $\mu_{|Y}$ is a contact form and that there exists a
unique vector field $\zeta \in \mathfrak{X}(Y)$, the
\textit{characteristic vector field} of the contact structure, such
that $\mu_{|Y}(\zeta) =1$ and $i_\zeta\,d\mu_{|Y} =0$. In
particular, the integral curves of the Euler-Lagrange system
coincide with the characteristic curves of $\zeta$.

\end{remark}

\subsection{The natural equation of integral curves}\label{ss:e-l-eq}

Let $\mathcal{V}(\mathcal{J},\omega)$ be the set of integral curves
of the Euler-Lagrange Pfaffian system $(\mathcal{J},\omega)$. If
$y=((\Gamma,k);x_1,x_2) : I \to Y$ is in
$\mathcal{V}(\mathcal{J},\omega)$, then equations
\[
 \eta^1=\eta^2 =\cdots =\eta^5 =0
  \]
and the independence condition $\omega \neq 0$ tell us that $\Gamma$
defines a canonical frame along the null curve $\gamma = \Gamma J
\Gamma^\ast$ and that $k$ is the curvature of $\gamma$.

Next, for the smooth function $k : I \to \R$, let $k'$, $k''$ and
$k'''$ be defined by
\[
 dk = k'\omega,\quad dk' = k''\omega, \quad dk'' = k''' \omega.
    \]
Equation $\sigma_3=0$ implies
\[
 x_1 = -\frac{\varepsilon}{2}k'.
  \]
Further, equation $\sigma_1=0$ gives
\[
 x_2 = \frac{1}{4}k''-\frac{\varepsilon}{2}(k^2 -mk +2).
  \]
Finally, equation $\sigma_2=0$ yields
\begin{equation}\label{euler-lagrange-eq}
  k''' -6\varepsilon k k' +2 \varepsilon m k' =0.
   \end{equation}
This is the Euler--Lagrange equation of the extremals of
\eqref{action}. It has been computed for example in
\cite{FGL-PhysLettB}. Thus, an integral curve of the Euler--Lagrange
system projects to an extremal trajectory in $\mathbb{S}^3_1$.

Conversely, let $\gamma : I \to \mathbb{S}^3_1$ be a null curve
without flex points, $\Gamma_\gamma$ its canonical frame and
$k_\gamma$ its curvature. Define the lift $y_\gamma : I \to Y$ of
$\gamma$ to the momentum space $Y$ by
\[
  y_\gamma(t) = \left((\Gamma_\gamma, k_\gamma); -\frac{\varepsilon}{2}k_\gamma',
   \frac{1}{4}k_\gamma''-\frac{\varepsilon}{2}(k_\gamma^2 -mk +2)\right).
    \]
Then, $y_\gamma$ is an integral curve of the Euler--Lagrange system
if and only if $k_\gamma$ satisfies equation
\eqref{euler-lagrange-eq} if and only if $\gamma$ is an extremal
trajectory. Thus, the integral curves of the Euler--Lagrange system
arise as lifts of trajectories in $\mathbb{S}^3_1$.

\subsection{The Lax formulation}

Introduce the \textbf{reduced curvature}
\[
  h := \frac{\varepsilon}{2}\left(k - \frac{m}{3}\right)
   \]
and identify $Y\cong\SL(2,\C)\times \R^3$, where $\R^3$ has
coordinates $(h,h',h'')$. Then, the Pfaffian equations defining the
the Euler--Lagrange system $\mathcal{J}$ are given by
\begin{equation}
 \left\{\begin{array}{l}
  \eta^j = 0, \quad (j=1,\dots,5)\\
     dh = h'\omega,\\
     dh' = h''\omega,\\
        dh'' = 12 hh'\omega,
       \end{array}
        \right.
        \end{equation}
where $\omega \neq 0$ is the independence condition. Equation
\eqref{euler-lagrange-eq} becomes
\begin{equation}\label{potential-eq.}
  h''' - 12 h h' = 0,
   \end{equation}
\begin{equation}
  H(h) = \left(
   \begin{array}{cc}
     0&\varepsilon\\
      2\varepsilon h + \frac{m}{3} +i&0\\
       \end{array}\right)
        \end{equation}
and
\[
\begin{split}
 \mu &= -\left(\varepsilon h -\frac{m}{3}\right)\beta^2_1 -h'\beta^1_1
 +\frac{\varepsilon}{2}
  \left(h'' -4h^2 +\frac{2}{3}\varepsilon m h + \frac{2}{9} m^2 -2\right)\beta^1_2\\
   &\quad+\left(h + \frac{2}{3}\varepsilon m\right)\alpha^1_2 + \alpha^2_1.
    \end{split}
     \]
Next, define the \textit{momentum} associated with $h$, ${U}(h)
\in\sl(2,\C)$, by
\begin{equation}\label{momento}
 \left(
   \begin{array}{cc}
     ih'&
      2i\varepsilon\left(h - \varepsilon\left(\frac{m}{3} + i\right)\right)\\
       2\left(h + \frac{2\varepsilon m}{3}\right) - i\varepsilon
       \left(h'' -4h^2 +\frac{2\varepsilon m h}{3} +
        \frac{2 m^2}{9} -2\right)&
          -ih'\\
         \end{array}
          \right).
           \end{equation}

A direct computation shows that equation \eqref{potential-eq.} is
equivalent to
\[
 {U(h)}' = \left[ U(h), H(h)\right].
  \]
The above discussion yields the following result.

\begin{prop}
A map $(A;h,h',h'') : I \subset \R \to Y$ is an integral curve of
the Euler--Lagrange system $(\mathcal{J},\omega)$ if and only if
\begin{equation}\label{lax}
 \left\{
  \begin{aligned}
   & A^{-1} A' = H(h),\\
    & {U(h)}' = \left[ U(h), H(h)\right].
    \end{aligned}
    \right.
     \end{equation}
\end{prop}

As a consequence, we have

\begin{cor}
The momentum map
\[
 \Phi : Y \to \sl(2,\C),\quad
    (A;h,h',h'') \mapsto A U(h) A^{-1}
     \]
is constant on integral curves of the Euler--Lagrange system.
\end{cor}

\begin{remark}\label{moment-map}
The momentum space $Y$ may be identified with $\SL(2,\C)\times \a$,
where $\a =\span\{U(h)\}$ is an affine subspace of $\sl(2,\C)$. The
group $\SL(2,\C)$ acts on $(Y, \mu)$ by
\[
  g\cdot (A; U(h)) = (gA; U(h)), \quad \text{for each }\, g\in \SL(2,\C),\,
U(h) \in \a,
  \]
in a Hamiltonian way. Using the isomorphism of $\sl(2,\C)$ with its
dual Lie algebra induced by the Killing form, one sees that the
momentum map associated with this action is given by $\Phi$. Moreover,
if $y = (A(t), U(h)(t))$ is an integral curve of the characteristic
vector field $\zeta$, then $U(h)(t)$ is an integral curve of the
vector field
\[
  X_\zeta : U(h) \mapsto \left[ U(h), H(h)\right]
   \]
and $\zeta$ can be written
\[
 \zeta_{|y} = H(h)_{|A} + X_\zeta (U(h)),
  \]
for all $y=(A, U(h)) \in Y$. If $\a_s$ denotes the {singular set} of
$X_\zeta$, then the integral curves through $(A, U(h))\in \SL(2,\C)
\times \a_s$ are orbits of the 1-parameter subgroups generated by
$H(h)$. By \eqref{momento}, these project to curves with constant
curvature (null helices). Next, consider $\Phi : \SL(2,\C) \times
\a_r \to \sl(2,\C)$, where $\a_r$ denotes the complement of $\a_s$
in $\a$. For each regular value $\ell\in \sl(2,\C)$ of $\Phi$, the
isotropy subgroup at $\ell$, $\SL(2,\C)_\ell$, is {abelian} and
$\dim \SL(2,\C)_\ell = \text{rank}\,\SL(2,\C)_\ell = 2$. The reduced
space $Y_\ell = \Phi^{-1}(\ell)/\SL(2,\C)_\ell$ is then
1-dimensional. This implies that an integral curve $y$ with momentum
$\ell$ (i.e., $\Phi \circ y = \ell$) can be found by quadratures.
Any other integral curve with momentum $\ell$ is given by $b\cdot
y$, for some $b\in \SL(2,\C)_\ell$.

Note that when the action of the symmetry group on the momentum
space is co-isotropic (as in the present case), the equation
governing the flow of $X_\zeta$ can always be written in Lax form.
See, for instance, \cite{GM-JGP}.
\end{remark}

\section{Integration of the trajectories}\label{s:integration}

\subsection{Preparatory material}

From equation \eqref{potential-eq.}, it follows that the reduced
curvature $h$ satisfies
\begin{equation}\label{weierstrass-eq}
 ({h}')^2 = 4h^3 -g_2h -g_3,
  \end{equation}
for real constants $g_2$ and $g_3$. Hence $h$ is expressed by the
real values of either a Weierstrass $\wp$-function with invariants
$g_2$, $g_3$, or one of its degenerate forms.

\vskip0.2cm

We call a solution to \eqref{weierstrass-eq} a \textit{potential}
with analytic invariants $g_2$, $g_3$. Two potentials are considered
equivalent if they differ by a re-parametrization of the form $s
\mapsto s+c$, where $c$ is a constant.\footnote{When invariants
$g_2$ and $g_3$ are given, such that $27g_3^2 \neq g_2^3$, the
general solution of the differential equation $(\frac{dy}{dz})^2 =
4y^3 -g_2y -g_3$ can be written in the form $\wp(z
+\alpha;g_2,g_3)$, where $\alpha$ is a constant of integration.} For
real $g_2$ and $g_3$, let $\Delta(g_2,g_3) = 27g_3^2 -g_2^3$ be the
discriminant of the cubic polynomial
\[
 P(t;g_2,g_3) = 4t^3 -g_2t -g_3.
  \]
The study of the real values of the Weierstrass $\wp$-function with
real invariants $g_2$, $g_3$ (and its degenerate forms) leads to
primitive half-periods $\omega_1$, $\omega_3$ such that (see for
instance \cite{Lawden}):
\begin{itemize}

\item $\Delta(g_2,g_3)<0$: $\omega_1>0$, $\omega_3=i\nu\omega_1$, $\nu>0$.

\item $\Delta(g_2,g_3)>0$: $\omega_1>0$, $\omega_3=\frac{1}{2}(1+i\nu)\omega_1$,
$\nu>0$.

\item $\Delta(g_2,g_3)=0$ and $g_3 > 0$: $\omega_1>0$, $\omega_3 =+i\infty$.

\item $\Delta(g_2,g_3)=0$ and $g_3 < 0$: $\omega_1=+\infty$, $-i\omega_3>0$.

\item $g_2=g_3=0$: $\omega_1=+\infty$, $\omega_3=+i\infty$.

\end{itemize}
Accordingly, denoting by $\mathcal{D}(g_2,g_3)$ the fundamental
period-parallelogram spanned by $2\omega_1$ and $2\omega_3$, the
only possible cases for the potential function $h : I \to \R$ are:

\begin{itemize}
\item $\Delta < 0$:
$h(s) = \wp(s;g_2,g_3)$, $I=(0,2\omega_1)$.

\item $\Delta < 0$:
$h(s) = \wp_3(s;g_2,g_3) = \wp (s + \omega_3;g_2,g_3)$, $I=\R$.

\item $\Delta > 0$:
$h(s) = \wp(s;g_2,g_3)$, $I=(0,2\omega_1)$.

\item $\Delta =0$, $g_3 = -8a^3 > 0$:
\[
 h(s) = - 3a\tan^2{\left({\sqrt{-3a}}{s}\right)} -2a,
\quad I=(-\frac{\pi}{\sqrt{-12a}},\frac{\pi}{\sqrt{-12a}}).
     \]
\item $\Delta =0$, $g_3 = -8a^3 < 0$:
\[
 h(s) = 3a\tanh^2{\left({\sqrt{3a}}{s}\right)} -2a,
 \quad I=\R.
     \]

\item $g_2=g_3 =0$:
$h(s) = {s^{-2}}$, $I=(-\infty,0)$ or $I=(0,+\infty)$.

\end{itemize}

\vskip0.2cm

Let $h$ be a Weierstrass potential with real invariants $g_2$,
$g_3$, and $U(h)$ the corresponding momentum as given by
\eqref{momento}. Then
\begin{eqnarray*}
 \det U(h) &=&
  \left(\frac{4}{27}m^3 -4 m - \frac{m}{3}g_2
   -\varepsilon g_3\right) +i\varepsilon\left(\frac{4}{3} m^2
    - g_2 -4\right)\\
 &=& P\left(\varepsilon\left(\frac{m}{3} + i\right); g_2,g_3\right).
     \end{eqnarray*}
Let
\[
 \nu(m,h) : = \sqrt{P\left(\varepsilon\left(\frac{m}{3}
  + i\right); g_2,g_3\right)},
  \]
chosen once for all. Then $\pm\nu(m,h)$ are the eigenvalues of the
momentum $U(h)$.

Next, define
\begin{equation}\label{third-kind-ints}
\phi(m,h) :=
\begin{cases}
 \displaystyle\int{\frac{\nu(m,h)}{h - \varepsilon\left(\frac{m}{3} + i\right)}ds},
  \quad \nu(m,h)\neq 0,\\
   \displaystyle\int{\frac{1}{h - \varepsilon\left(\frac{m}{3} + i\right)}ds},
    \quad \nu(m,h)= 0.
     \end{cases}
       \end{equation}
These are elliptic integrals of the third kind. Let $w(m,h)$ be the
unique point in the period-parallelogram $\mathcal{D}(g_2,g_3)$ such
that
\[
  h(w) = \varepsilon\left(\frac{m}{3} + i\right) \quad\text{and}
  \quad   h'(w) = \nu(m,h).
   \]
Denote by $\sigma_h$ and $\zeta_h$, respectively, the sigma and zeta
Weierstrassian functions corresponding to the potential $h$, i.e.,
the unique analytic odd functions whose meromorphic extensions
satisfy $\zeta'_h = -h$ and $\sigma_h'/\sigma_h = \zeta_h$. Under
the above assumptions, we now compute the elliptic integrals
\eqref{third-kind-ints}. Three cases are considered.

\vskip0.2cm \noindent \textbf{Case I:} $\nu(m,h) \neq 0$. In this
case,
\[
 \phi(m,h) = \int{\frac{h'(w)}{h(s) - h(w)}ds}   =  \log{\frac{\sigma_h(s-w)}{\sigma_h(s+w)}} +
   2s\zeta_h(w) + \text{const}.
    \]

\vskip0.2cm \noindent \textbf{Case II:} $\nu(m,h) = 0$ and $g_2^2
+g_3^2 \neq 0$. In this case, $h(w) = \varepsilon\left(\frac{m}{3} +
i\right)$ is a root of the cubic polynomial $P$, say $e_3$. If $e_1,
e_2$ denote the other two roots, we have
\begin{eqnarray*}
 \phi(m,h) &=& \int{\frac{ds}{h(s) - e_3}} =\int{\frac{h(s+w)-e_3}{(e_3-e_1)(e_3-e_2)}ds}\\
 &=&\frac{1}{\displaystyle\frac{g_2}{4} -3\left(\frac{m}{3} + i\right)^2
  }\left\{\zeta_h(s + w) +
   \varepsilon\left(\frac{m}{3} + i\right)s\right\} + \text{const}.
    \end{eqnarray*}

\vskip0.2cm \noindent \textbf{Case III:} $\nu(m,h) = 0$ and
$g_2=g_3= 0$. In this case,
\[
 \phi(m,h) =\frac{1}{3}s^3 + \text{const}.
   \]

\subsection{Explicit integration}

We are now in a position to explicitly integrate the extremal
trajectories. This amounts to integrate by quadratures the reduced
system associated to the Hamiltonian action of $\SL(2,\C)$ on $Y$
(cf. Remark \ref{moment-map}). The key to explicit integration is
the conservation of the momentum map along integral curves of the
Euler-Lagrange system.

\begin{thm}\label{main}
Let $\gamma : I \to \mathbb{S}^3_1$ be an extremal trajectory with Lagrange
multiplier $m$
and reduced curvature $h$ with real invariants $g_2$, $g_3$. Let
$U_\gamma(h)$ be the momentum of $h$ given by \eqref{momento}, and
assume that $\gamma$ be parametrized by the canonical parameter $s$,
i.e., $\omega = ds$. According as $\det U_\gamma(h)$ is zero, or
different from zero, we distinguish two cases.

\vskip0.2cm

\noindent \textbf{Case I:} If $\det U(h) \neq 0$, then the canonical
frame field $\Gamma : I \to \SL(2,\C)$ along $\gamma$ is given by
 \[
  \Gamma(s) = A\cdot M(s),
   \]
where $A\in \SL(2,\C)$ and $M(s)$ takes the form
\[
\frac{1}{\sqrt{-4i\varepsilon \nu}} \left(
     \begin{array}{cc}
     \displaystyle e^{\phi(m,h)}&
       \displaystyle 0\\
       \displaystyle 0  &
       \displaystyle e^{-\phi(m,h)}
          \end{array}\right)
\left(
 \begin{array}{cc}
\displaystyle \frac{ih'+\nu}{\sqrt{h - \varepsilon\left(\frac{m}{3}
+ i\right)}}&
 2i\varepsilon{\sqrt{h - \varepsilon\left(\frac{m}{3} + i\right)}}\\
\displaystyle \frac{-ih'+\nu}{\sqrt{h - \varepsilon\left(\frac{m}{3}
+ i\right)}}
 &
\displaystyle -2i\varepsilon{\sqrt{h - \varepsilon\left(\frac{m}{3}
+ i\right)}}
          \end{array}\right)
\]

\vskip0.2cm

\noindent \textbf{Case II:} If $\det U(h) = 0$, then the canonical
frame field $\Gamma : I \to \SL(2,\C)$ along $\gamma$ is given by
 \[
  \Gamma(s) = A\cdot M(s),
   \]
where $A\in \SL(2,\C)$ and $M(s)$ takes the form
\[
\frac{1}{\sqrt{-2i\varepsilon}} \left(
     \begin{array}{cc}
 \frac{1}{\sqrt{h - \varepsilon (\frac{m}{3} + i)}}&
 \frac{-\phi(m,h)}{2i}\\
        0  &
        1
          \end{array}\right)
\left(
 \begin{array}{cc}
  \displaystyle 1 &
   0\\
    \frac{-ih'}{\sqrt{h - \varepsilon (\frac{m}{3} + i)}} &
  -2i\varepsilon\sqrt{h - \varepsilon (\frac{m}{3} + i)}\\
          \end{array}\right)
\]
\end{thm}

\begin{proof}[Proof of Case I]
Let $\Gamma = (C_1,C_2) : I \to \SL(2,\C)$ be a canonical frame
along $\gamma$ and $U_\gamma(h)$ be the momentum of $\gamma$ given
by \eqref{momento}. Consider the eigenvalues $\pm \nu(m,h)$ of
$U_\gamma(h)$ and denote by $\mathbf{L}_{\pm}$ the corresponding
eigenspaces. From the definition of $U_\gamma(h)$, it follows that
\[
\begin{aligned}
 L_+ &= -2i\varepsilon\left(h - \varepsilon\left(\frac{m}{3} + i\right)\right)C_1
    + \left(ih' -\nu(m,h)\right)C_2 : I \to \mathbf{L}_+ \\
  L_- &= -2i\varepsilon\left(h - \varepsilon\left(\frac{m}{3} + i\right)\right)C_1
    + (ih' +\nu(m,h))C_2 : I \to \mathbf{L}_-
   \end{aligned}
    \]
are eigenvectors of $U_\gamma(h)$ corresponding to $\nu(m,h)$ and
$-\nu(m,h)$, respectively. Thus, we must have
\[
 L'_+ = \rho_1 L_+, \quad L'_-= \rho_2 L_-,
  \]
for analytic functions $\rho_1$, $\rho_2$. Using the Maurer--Cartan
equation $\Gamma' =\Gamma H(m,h)$, we compute
\[
  L'_+ = \frac{h' +
          \nu(m,h)}{2\left(h -
           \varepsilon\left(\frac{m}{3} + i\right) \right)} L_+,
    \quad
       L'_-= \frac{h' -
              \nu(m,h)}{2\left(h -
               \varepsilon\left(\frac{m}{3} + i\right) \right)} L_-.
        \]
We thus see that the two vectors
\[
 \begin{aligned}
 \Lambda_1 & := \exp{\left(-\int{\frac{h' + \nu(m,h)}{2\left(h -
           \varepsilon\left(\frac{m}{3} + i\right) \right)}ds}\right)}L_+
           \\
 \Lambda_2 & := \exp{\left(-\int{\frac{h' -
          \nu(m,h)}{2\left(h -
           \varepsilon\left(\frac{m}{3} + i\right) \right)}ds}\right)}L_-
            \\
  \end{aligned}
   \]
are constant along $\gamma$. By \eqref{third-kind-ints}, they become
\[
\Lambda_1 =
 \frac{\exp{\left(-\phi(m,h)\right)}}{\sqrt{h -\varepsilon\left(\frac{m}{3} + i\right)}}L_+,
            \qquad
 \Lambda_2 =
\frac{\exp{\left(\phi(m,h)\right)}}{\sqrt{h
-\varepsilon\left(\frac{m}{3} + i\right)}}L_-.
   \]
Hence
\[
  \Gamma\cdot R(m,h)\cdot S(m,h) = \Lambda = (\Lambda_1, \Lambda_2) \quad
   \]
where
\[
 R(m,h) =
     \left(
     \begin{array}{cc}
     \displaystyle -2i\varepsilon\left(h - \varepsilon\left(\frac{m}{3} + i\right)\right)&
       \displaystyle-2i\varepsilon\left(h - \varepsilon\left(\frac{m}{3} + i\right)\right)\\
       \displaystyle ih'-\nu(m,h)  &
         \displaystyle ih'+\nu(m,h)
          \end{array}\right).
            \]
and
\[
S(m,h) =\left(
     \begin{array}{cc}
      \frac{1}{\sqrt{h - \varepsilon (\frac{m}{3} + i)}}&
        0\\
        0  &
       \frac{1}{\sqrt{h - \varepsilon (\frac{m}{3} + i)}}
          \end{array}\right)
\left(
     \begin{array}{cc}
      \exp{(-\phi(m,h))}&
        0\\
        0  &
        \exp{(\phi(m,h))}
          \end{array}\right),
         \]
From this, we obtain
\[
\Gamma\left(
 \begin{array}{cc}
  -2i\varepsilon{\sqrt{h - \varepsilon (\frac{m}{3} + i)}}&
   -2i\varepsilon{\sqrt{h - \varepsilon (\frac{m}{3} + i )}}\\
   \frac{ih'-\nu}{\sqrt{h - \varepsilon (\frac{m}{3} + i)}} &
    \frac{ih'+\nu}{\sqrt{h - \varepsilon (\frac{m}{3} + i )}}\\
          \end{array}\right)
\left(
     \begin{array}{cc}
     \displaystyle e^{-\phi(m,h)}&
       \displaystyle 0\\
       \displaystyle 0  &
       \displaystyle e^{\phi(m,h)}
          \end{array}\right) = \Lambda
       \]
and hence
\[
\Gamma = \tilde{\Lambda} \left(
     \begin{array}{cc}
     \displaystyle e^{\phi(m,h)}&
       \displaystyle 0\\
       \displaystyle 0  &
       \displaystyle e^{-\phi(m,h)}
          \end{array}\right)
\left(
 \begin{array}{cc}
 \displaystyle \frac{ih'+\nu}{\sqrt{h - \varepsilon\left(\frac{m}{3} + i\right)}}&
  \displaystyle 2i\varepsilon{\sqrt{h - \varepsilon\left(\frac{m}{3} + i\right)}}\\
  \displaystyle \frac{-ih'+\nu}{\sqrt{h - \varepsilon\left(\frac{m}{3} + i\right)}}
   &
  \displaystyle -2i\varepsilon{\sqrt{h - \varepsilon\left(\frac{m}{3} + i\right)}}
          \end{array}\right).
    \]
\end{proof}

\begin{proof}[Proof of Case II]
Again, let $\Gamma = (C_1,C_2) : I \to \SL(2,\C)$ be a canonical
frame along $\gamma$ and $U_\gamma(h)$ be the momentum of $\gamma$.
If $\nu(m,h) =0$, then
\[
 L_1 = -2i\varepsilon\left(h - \varepsilon\left(\frac{m}{3} + i\right)\right)C_1
    + ih' C_2  \\
    \]
belongs to the kernel of $U_\gamma(h)$, and proceeding as in Case I,
we see that the vector
\begin{equation}\label{first-int-1}
 \Lambda_1 =
  \frac{{1}}{\sqrt{h -\varepsilon\left(\frac{m}{3} + i\right)}}L_1
    \end{equation}
is a first integral. In order to find another first integral, we
look for analytic functions $f$ and $g$ such that
\begin{equation}\label{first-int-2}
 \Lambda_2 := gC_2 + fL_1
  \end{equation}
be a constant vector. Differentiating and using the Maurer--Cartan
equation $\Gamma' = \Gamma H(m,h)$, we obtain
\[
 g'C_2 + g\varepsilon C_1 + f' L_1 = 0,
 \]
from which we compute
\[
 g = \frac{{1}}{\sqrt{h -\varepsilon\left(\frac{m}{3} + i\right)}}, \quad
  f = \frac{1}{2i}\int{\frac{ds}{h - \varepsilon\left(\frac{m}{3} + i\right)}}=
  \frac{1}{2i}\phi(m,h).
   \]
Now, from \eqref{first-int-1} and \eqref{first-int-2}, we obtain
\[
\Gamma\left(
 \begin{array}{cc}
  \displaystyle-2i\varepsilon\sqrt{h - \varepsilon (\frac{m}{3} + i)}&
   0\\
   \displaystyle\frac{ih'}{\sqrt{h - \varepsilon (\frac{m}{3} + i)}} & 1\\
          \end{array}\right)
\left(
     \begin{array}{cc}
       1&  \displaystyle\frac{1}{2i}\phi(m,h)\\
       \displaystyle 0  &
       \displaystyle \frac{1}{\sqrt{h - \varepsilon\left(\frac{m}{3} + i\right)}}
          \end{array}\right) = \Lambda =(\Lambda_1, \Lambda_2),
       \]
and hence the required result.

\end{proof}

\bibliographystyle{amsalpha}

\begin{thebibliography}{AA}

\bibitem{Br}
R. L. Bryant, {On notions of equivalence of variational problems
with one independent variable}, \textit{Contemp. Math.} \textbf{68}
(1987), 65--76.

\bibitem{BG}
 R. L. Bryant, P. A. Griffiths,
{Reduction for constrained variational problems and
$\int{{\frac{k^2}{2}}}$}, \textit{Amer. J. Math.} \textbf{108}
(1986), 525--570.

\bibitem{C1}
 E. Cartan,
\textit{Sur un probl\`eme du Calcul des variations en G\'eom\'etrie
projective plane}, Oeuvres Compl\`etes, Partie III, vol. 2,
1105--1119, Gauthier--Villars, Paris, 1955.

\bibitem{CartanLSLII}
 E. Cartan,
\textit{Le\c{c}ons sur les invariants int\'egraux}, Hermann, Paris,
1922.

\bibitem{Cas}
M. Castagnino, Sulle formule di Frenet-Serret per le curve nulle di
una $V\sb{4}$\ riemanniana a metrica iperbolica normale,
\textit{Rend. Mat. e Appl.} (5) \textbf{23} (1964), 438--461.

\bibitem{FGL-IJMPA}
A. Ferr\'andez, A. Gim\'enez, P. Lucas, Null helices in Lorentzian
space forms, \textit{Internat. J. Modern Phys. A} \textbf{16}
(2001), no. 30, 4845--4863.

\bibitem{FGL-PhysLettB}
A. Ferr\'andez, A. Gim\'enez, P. Lucas, Geometrical particle models
on 3D null curves, \textit{Phys. Lett. B} \textbf{543} (2002),
311--317; {\tt arXiv:hep-th/0205284.}

\bibitem{Gardner1983}
R. G. Gardner, Differential geometric methods interfacing control
theory; in \textit{Differential geometric control theory} (Houghton,
MI, 1982), R. W. Brockett, R. S. Millman and H. J. Sussmann (eds.),
117--180, Progr. Math., 27, Birkh\"auser, Boston, 1983.

\bibitem{Gardner1989}
R. G. Gardner, The method of equivalence and its applications,
CBMS-NSF Regional Conference Series in Applied Mathematics, 58,
SIAM, Philadelphia, 1989.

\bibitem{GM-JGP}
J. D. E. Grant, E. Musso, Coisotropic variational problems,
\textit{J. Geom. Phys.} \textbf{50} (2004), 303--338; {\tt
arXiv:math.DG/0307216.}

\bibitem{Gr}
 P. A. Griffiths,
\textit{Exterior differential systems and the calculus of
variations}, Progr. Math., 25, Birkh\"auser, Boston, 1982.

\bibitem{Hsu}
  L. Hsu,
{Calculus of variations via the Griffiths formalism}, \textit{J.
Differential Geom.} \textbf{36} (1992), 551--589.

\bibitem{LS}
 J. Langer, D. Singer,
{Liouville integrability of geometric variational problems},
\textit{Comment. Math. Helv.} \textbf{69} (1994), 272--280.

\bibitem{Lawden}
D. F. Lawden, \textit{Elliptic functions and applications}, Applied
Mathematical Sciences, 80, Springer-Verlag, New York, 1989.

\bibitem{MN-Forum}
E. Musso, L. Nicolodi, Reduction for the projective arclength
functional, \textit{Forum Math.} \textbf{17} (2005), 569-590.

\bibitem{MN-adS}
E. Musso, L. Nicolodi, Closed trajectories of a particle model on
null curves in anti-de Sitter 3-space, \textit{Classical Quantum
Gravity} (to appear); {\tt arXiv:0709.2017 [math.DG].}

\bibitem{NMMK-NucPhysB}
A. Nersessian, R. Manvelyan, H. J. W. M\"uller-Kirsten, Particle
with torsion on 3d null-curves, \textit{Nuclear Phys. B} \textbf{88}
(2000), 381--384; {\tt arXiv:hep-th/9912061.}

\bibitem{NR-PhysLettB}
A. Nersessian, E. Ramos, Massive spinning particles and the geometry
of null curves, \textit{Phys. Lett. B} \textbf{445} (1998),
123--128; {\tt arXiv:hep-th/9807143.}

\bibitem{P-NucPhysB}
M. S. Plyushchay, The model of the relativistic particle with
torsion, \textit{Nucl. Phys. B} \textbf{362} (1991), 54--72.

\bibitem{Pi}
E. Picard, {Sur les \'equations diff\'erentielles lin\'eaires \`a
coefficients dou\-ble\-ment p\'eriodiques}, \textit{J. Reine Angew.
Math.} \textbf{90} (1881), 281--302.


\end{thebibliography}

\end{document}